\documentclass[12pt]{article}

\usepackage[margin=1.2in]{geometry}  
\usepackage{graphicx}              
\usepackage{amsmath}               
\usepackage{amsfonts}              
\usepackage{amsthm}                
\usepackage{hyperref}

\newtheorem{thm}{Theorem}[section]

\begin{document}

\nocite{*}

\title{Geometric deep learning approach to knot theory}

\author{Lennart Jaretzki\\ 
lennart.jaretzki@gmail.com \\
Leipzig Germany \\
}

\maketitle

\begin{abstract}
 
In this paper, we introduce a novel way to use geometric deep learning for 	knot data by constructing a functor that takes knots to graphs and using graph neural networks. We will attempt to predict several knot invariants with this approach. This approach demonstrates high generalization capabilities.

\end{abstract}

\section{Introduction}
Recent papers have researched the application of machine learning and neural networks to knot theory. One common approach to learning knots is to use their braid representation as input instead. This has the downside that extra crossings have to be added. Also, it makes the data less dense, and the neural network has to abstract the underlying knot. The neural networks used are often dense models that operate on a fixed-sized input, sequence-to-sequence models, transformers, and other natural language processing approaches \cite{8, untangleBraid, learningToUnknot}. There was some research conducted on using a more geometric approach by using rectangular diagrams, but they used LSTMs instead of a geometric deep learning approach, and in the data representation, arbitrary choices had to be made \cite{rectangularKnotLearning}. It seems much more intuitive to use methods from geometric deep learning to learn knot data \cite{GDL}. We address these issues by proposing a method to use graph neural networks to learn knot data. For this purpose, we construct a functor that takes knots to graphs, which are learned by graph neural networks. We will attempt to predict both geometric and combinatorial invariants. We will conduct experiments using the graph transformer operator. The code for this project can be found on \href{https://github.com/LennartJaretzki/geometricDeepLearningKnotData}{github}.

\section{Graph construction}
Using graph neural networks for knot learning requires a way to represent a given knot as a graph. One method to do this would be by using black graphs, but they suffer from the problem that the graph is often a self-connected graph, i.e. it contains loops. Another problem is that a large number of black graphs are only pseudographs. This makes it hard to apply graph neural networks to them. Our proposed functor is similar and takes knots to graphs by taking the faces of the graph to the vertices of the graph and taking edges between crossings to perpendicular graph edges. This induces taking the crossings of the knot to the faces of the resulting graph. By definition, the resulting graph is planar. This brings up the question if there is a unique way to reconstruct the knot from a given knot. Such that a unique reconstruction is possible, the edges of the resulting graph need to have two edge attributes. The first attribute is a measure of distance that is necessary to preserve the general structure of the knot, and the second edge attribute is a boolean that indicates if a given edge between crossings is altering between "over" and "under". Since only the alteration is given, it is not possible to differentiate chirality and handedness. Note that this procedure generalizes to links. It is probably also possible to generalize the procedure to virtual knots.

\begin{figure}[ht]
\begin{center}
\includegraphics[width=2.2in]{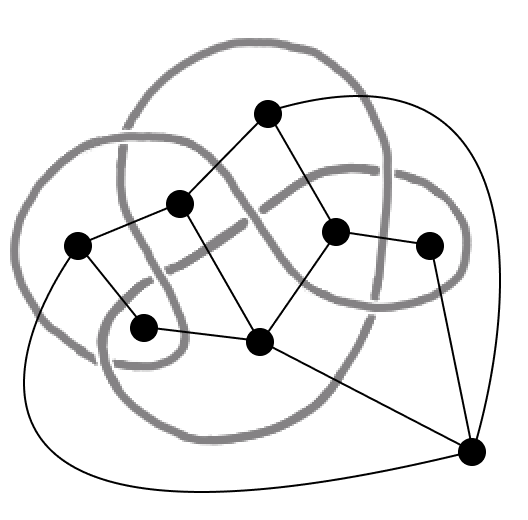}
\end{center}
\caption{The described Functor at the example of the $6\_3$ with the crossing in the middle flipped.\label{exampleFunctor}}
\end{figure}

\subsection{Rigorous definition}
Given a knot, first all edges of the knot diagram need to be enumerated, i. e. labeled with increasing numbers (increasing while we go along the paths that connect two crossings). Next, the alteration of each edge needs to be determined. Now that we have gathered this information, we can transform our knot into a graph.

\begin{equation}
V:=\lbrace F\vert \partial F = \bigcup\limits_{i\in A}e_i\text{ with }\forall i\in A: e_i \text{ is edge in the knot diagram} , F\text{ is connected region}\rbrace\,.
\label{graphVertecies}
\end{equation}

\begin{equation}
E:=\lbrace (F_0, F_1)\vert F_0, F_1 \in V, F_0\cap F_1 \neq \emptyset \rbrace\,.
\label{graphEdges}
\end{equation}
Note that this has a similar problem to black graphs because two faces may have a common edge in two or more distinct positions. This is relatively rare and did not happen once in the dataset. It does occur when the knot gets twisted or two knots get added.

\begin{equation}
G = (V, E).
\label{graph}
\end{equation}
Let $e\in E$:

\begin{equation}
l(e):=\text{label of the edge}
\label{labelMeasure}
\end{equation}
Note that by definition, if and only if $a, b \in E$, and $a$ and $b$ are on the opposite side of the same crossing, then:

\begin{equation}
l(a)\equiv l(b) \pm 1 \mod{2\cdot \text{\#Crossings(K)}}
\label{labelCrossingDefinition}
\end{equation}

\subsection{Restore knot}
To restore the knot, it is necessary to find a planar embedding of the knot graph.
\noindent
\begin{equation}
p: V \to \mathbb{R}^2
\label{Embedding}
\end{equation}
Such that $p$ is planar, i. e. there are no two graph edges that cross each other. To be able to reconstruct $K$ uniquely, we need the additional condition that, given an arbitrary graph face, this face is composed of exactly four edges. It is also necessary that opposing edges fulfill ~``Eq.~(\ref{labelCrossingDefinition})''. To reconstruct the knot diagram, every face gets mapped to a crossing. The outgoing edges of the constructed knot crossing get connected through the graph edges they are affiliated with, such that they connect with the crossings created by the adjoint graph faces. The last step is to assign an arbitrary crossing where the in-going path goes "under" After that, you can derive the state of the other crossings via the second edge attribute.

\begin{thm}
Two embeddings $p_0$, $p_1$ of $G$ restore ambient isotopic knots.
\end{thm}
\begin{proof}
Let $v_0 \in V$ be an arbitrary vertex. Next, $p_0(G)$ and $p_1(G)$ need to be embedded in $\mathbb{R}^3$. Define $w_p(v):=(p(v), d(v_0,v))$ where $d$ is a distance measure (the number of edges of the shortest path connecting $v_0$ and $v$). We can now transform $w_{p_0}(G)$ to $w_{p_1}(G)$ by shifting the vertecies. Define $u_t(w_{p}(g)):=(1-t)\cdot w_{p_0}(g) +t\cdot w_{p_1}(g)$ with $0 \leq t \leq 1$. We can find a corresponding knot embedded in $\mathbb{R}^3$ for each $u_t$ by embedding the edges of $G$ as straight lines and then embedding the crossings of the knot in the center of the faces that fulfill ~``Eq.~(\ref{labelCrossingDefinition})'' while letting the knot edges intersect with the midpoint of the embedded graph edges. This procedure is continuous along $t$ \footnote{You may have to find a homeomorphism of the embedding such that injectivity is preserved throughout the entire process.}. Since we can also lift the two graph embeddings continuously. Applying $w_{p_1}^{-1} \circ u_1 \circ w_{p_0}$ induces ambient isotopy \footnote{Note that the face may be flipped when applying $u$ but since we only have information about the crossings in relation to each other, this makes no difference.}.
\end{proof}

\section{Dataset}
The foundation of the dataset consists of 2977 knots from the Knotinfo Database \cite{knotinfo}. These consist of all knots up to 12 crossings.
\subsection{Data augmentation}
Since the purpose is to predict invariance, we can easily extend the dataset by applying random Reidemeister moves to the given knots. The exact algorithm used for shuffling the knots is described in \ref{shuffle}. This returns a new randomly shuffled knot, which is ambient isotopic to the original knot and is non-simplifiable, i. e. can't be solved by just undoing the first and second reidemeister moves. For every knot, there are 89 shuffled versions of the knot, such that the entire dataset of shuffled knots consists of 264,953 knots. The mean number of crossings of these knots is 34.56, with a standard derivation of 11.38. Also, a dataset with 2977 knots that have at least 70 crossings was generated. Later, this dataset will be used to test how well the neural networks can generalize to bigger knots. The mean number of knot crossings in this dataset is 91.3, with a standard deviation of 14.91.
\subsection{Validation datasets}
One validation dataset is generated by a test train split, with $20\%$ of the data being used for validation. We will also attempt to measure the generalization capabilities by training only on knots that have 11 or fewer crossings when solved and testing on knots that have 12 crossings when solved. As well as testing on the knots that have at least 70 crossings.
\subsection{Data preperation}
We need to address two issues before we can use our data as input for neural networks. First of all, in the distance measure, there was an arbitrary edge chosen to enumerate the edges. To work around the problem, we instead calculate the distance between the adjacent knot edges. Because neural networks work better with smooth values that lie in a predefined interval, we also apply an activation function to the distances.

\begin{figure}[ht]
\begin{center}
\includegraphics[width=2.2in]{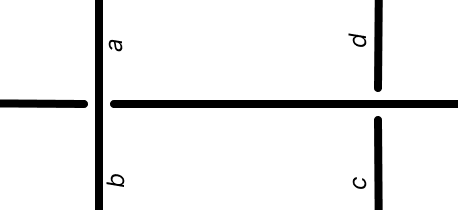}
\end{center}
\caption{Two crossings connected by one edge. The labels of the enumerated edges are $a,b,c,d$. {\it T}.\label{fig1}}
\end{figure}

The edge attribute that can be used by the neural network is given by

\noindent
\begin{equation}
1-2^{-\frac{d}{d_s}}.
\label{graphEdges}
\end{equation}
with $d \equiv \frac{\lvert(a+b-(c+d))\lvert}{2} \mod{\#Crossings(K)}$ being the distance, i. e. the minimal number of edges you need to walk to get from one to the other. If $d$ is higher than $d_s$ the edge attribute will be higher than $\frac{1}{2}$ and gets squished, while values lower than are more distinguishable and can therefore be processed easier. This may be advantageous because edges that are closer together are more intertwined and are easier to process. $d_s=15$ was arbitrarily chosen. Also, to make the edge attribute altering usable for the network, $altering$ and $not altering$ get mapped to $1$ and $-1$ respectively. 

\section{Experiments}
\subsection{Models}
Since the information about the knot is given by the edge attributes, it is necessary to choose a graph neural network architecture that can process edge attributes. One popular approach is graph convolution networks, which work by aggregating information from neighboring nodes and edges in a graph using a learned convolution operation to produce a new representation for each node in the graph \cite{PNA}. The transformer is a powerful architecture that utilizes attention mechanisms and has been adeptly used in the graph transformer \cite{graphTrans, GAT, vaswani2017attention}. We observed that the results were best when using the graph transformer. We apply the layer, which we batch normalize, and apply the $tanh$ activation function on this operation, which is then stacked four times. Since we want to make graph-level predictions, we then use global pooling, which means that we calculate the $mean$, $max$, and $min$ of all node features of the graph. These are then concatenated with 10 additional features of the given knot \footnote{Similar to \cite{8} we use the following features: alternating, fibered, positive braid closure, small or large, crossing number, signature, arc index, determinant, and Rasmussen invariant }. They are then again normalized, which increases results when generalizing to larger knots. These are the same as in the given tensor, which is then the input for a 4-layer multi-layer perceptron with 25 hidden neurons per layer. And one neuron that is predicting the feature. Then the mean absolute error will be calculated, and the model parameters will be updated using backpropergation and the Adam optimizer.

\subsection{Training}
The model will be fitted to the data over multiple epochs with a learning rate of $0.001$. Comparison of different graph neural networks fitting different knot invariants. The models were trained, and implemented using pytorch, and pytorch geometric \cite{pytorchgeometric}.

\subsection{Results}
The models have been trained over 200 epoches, and the one with the highest accuracy has been picked. In the case of volume, the mean distance to the correct result is displayed.
\begin{center}
\begin{tabular}{|c||c|c|c|c|}
\hline
Feature & used Additional & Validation Dataset & Large Knot Dataset & $\le$ 11 Results \footnotemark  \\
\hline
Q-Positivity & Yes & 92.38\% & 91.57\% & 91.13\% \\
Rasmussen-S & Yes & 95.63 & 95.36\% & - \\
Volume & Yes & 0.5498 & 0.7905 & 0.8485 \\
Volume & No & 2.0442 & 2.6864 & - \\

Ozsvath-Szabo $\tau$ & Yes & 99.78\% & 98.79\% & 99.92\% \\

\hline
\end{tabular}
\end{center}

\footnotetext{When the model is trained only on knots that have crossings that are less than or equal to 11 when simplified and then tested on knots that have 12 crossings when simplified.}
We can observe that even though the accuracy is not that high, the networks demonstrate high generalization both to larger knots and to knots that have more crossings when simplified. Also note that it works better on geometric/hyperbolic than on numerical/combinatorical \cite{hyperbolic, combi}.

\section{Similar work}
While conducting this research, a similar approach has been published that utilizes geometric deep learning to predict if two given 3-manifolds are homeomorphic \cite{plumbing}. This approach uses plumbing graphs as input. They also use techniques from reinforcement learning and graph attention networks. 

\appendix

\section{Shuffle algorithm}

Given a knot and a complexity $c$ the here used shuffle algorithm starts by applying $2\cdot c$ times random type 1 and type 2 reidemeister moves with probabilities of $20\%$ and $80\%$ respectively. After this, it applies $5$ random type 1 and 2 reidemeister moves with the same probabilities, after which $\lfloor \frac{c}{20} \rfloor$ random type 3 reidemeister moves are applied. This process is repeated $c$ times. After this, all excess type 1 and 2 reidemeister moves are reversed. \label{shuffle} Increasing $c$ increases the probability of getting bigger and more complex knots. This algorithm was implemented using \cite{snappy}.



\end{document}